\documentclass{cimart}

\DeclareMathOperator{\area}{area}
\DeclareMathOperator{\GL}{GL}
\DeclareMathOperator{\Isom}{Isom}
\DeclareMathOperator{\K}{Kiss}
\DeclareMathOperator{\PSL}{PSL}
\DeclareMathOperator{\SL}{SL}
\DeclareMathOperator{\SO}{SO}
\DeclareMathOperator{\sys}{sys}
\DeclareMathOperator{\tr}{tr}
\DeclareMathOperator{\vol}{vol}

\newcommand{\bs}{\backslash}
\newcommand{\C}{\mathbb{C}}
\renewcommand{\H}{\mathbb{H}}
\newcommand{\N}{\mathbb{N}}
\newcommand{\Q}{\mathbb{Q}}
\newcommand{\R}{\mathbb{R}}
\newcommand{\Z}{\mathbb{Z}}

\title{On systole, kissing number, and volume of arithmetic manifolds}

\authors{Plinio Murillo}

\authorinfo{Fluminense Federal University, Brazil}{pliniom@id.uff.br}

\abstract{The purpose of this expository article is to give a down-to-Earth introduction to the notion of an arithmetic group and an arithmetic manifold. To achieve this, we have decided to bring two geometrical questions relating to the growth of systoles and kissing numbers in hyperbolic manifolds, as a motivational guide, whose answers so far are given by the use of arithmetic manifolds. As a consequence, we answer these questions in detail for dimension 2, we mention what is known for hyperbolic manifolds of higher dimension, and also for other locally symmetric spaces. We end the exposition with some open questions.}



\acknowledgments{%
    I would like to thank the students and colleagues of Universidad de la República - Uruguay, who attended the minicourse \textit{On length of closed geodesics of arithmetic hyperbolic surfaces} in March 2025, where the first version of this article was prepared. Special thanks to León Carvajales, Sébastien Alvarez, Pablo Lessa, and Rafael Potrie for the invitation and hospitality during my time in Montevideo. I am also grateful to Rafael Potrie for the invitation to submit a survey to the Special Issue on Latin American mathematics of Communication in Mathematics. I would also like to thank Cayo Dória for his comments, suggestions, and collaboration. Many of the works  referred in this article have come to my attention during our fruitful collaboration during the last several years. Sincere thanks to the anonymous referees, whose comments and corrections helped to improve the exposition of the article. This work has been partially supported by Universal CNPq Grant 408834/2023-4.
    }

\VOLUME{34}
\ISSUE{2}
\NUMBER{13}
\DOI{http://doi.org/10.46298/cm.16457}
\editinfo{September 3, 2025}{April 7, 2026}{Tiago Macedo, Jaqueline Mesquita, Rafael Potrie, Mariel Saez}

\begin{document}

\section{Introduction}

The study of arithmetic groups is an important area of research with different and interesting connections with other branches of mathematics, such as
differential geometry, number theory, ergodic theory, dynamical systems, among many others. Some examples of this fruitful interaction can be found in the solution of long-standing open problems in number theory, such as the Oppenheimer Conjecture due to Margulis, and the understanding of the structure of unipotent flows by Ratner (e.g. \cite{Mar97} and \cite{Mor05}).

The main developments in the theory of arithmetic groups have been carried out mainly in research centers in Europe and United States. However, the interest in these topics has increased among mathematicians based in Latin America, particularly in connections with Riemannian geometry, ergodic theory, and dynamical systems, traditional areas of research in countries such as Argentina, Brazil, Mexico, and Uruguay. This increasing interest creates a great opportunity to connect a well-established area of research, the theory of arithmetic groups, to problems that are of interest to mathematicians in this part of the world. To accomplish this, we first need to heal a natural challenge: \textit{how can we help students and researchers, from different areas of mathematics, to understand the basic elements of the theory of arithmetic groups, usually seen as very technical?}

The purpose of this expository article is to give a down-to-Earth introduction to the notion of an arithmetic group and an arithmetic manifold. To achieve this we have decided to bring two geometrical questions relating to the growth of systoles and kissing numbers in hyperbolic manifolds, as a motivational guide, whose answers so far are given by the use of arithmetic manifolds. As a consequence, we answer these questions in detail for dimension 2, we mention what is known for hyperbolic manifolds of higher dimension, and also for other locally symmetric spaces. We end the exposition with some open questions.

\section{Two guiding questions}

\subsection{Hyperbolic manifolds} A \textit{hyperbolic $n$-manifold} is a differentiable manifold of dimension $n\geq 2$ together with a complete Riemannian metric of constant sectional curvature equal to $-1$. The upper-half model of the hyperbolic $n$-space is given by
\[\H^n=\left\{x=(x_0,x_1,\ldots,x_n)\in\R^{n+1}| x_n>0\right\}\]
together with the Riemannian metric $ds^2=\frac{dx_0^2+dx_1^2+\cdots+dx_n^2}{x_n^2}$. The $n$-manifold $\H^n$ is the unique simply connected manifold with sectional curvature $-1$ (up to isometry). The isometry group $\Isom(\H^n)$ is isomorphic to the Lie group $\SO(n,1)$. So, $M$ is a hyperbolic $n$-manifold if and only if there is a subgroup $\Gamma<\SO(n,1)$ acting proper discontinuously on $\H^n$ and without fixed points, such that $M=\Gamma\bs\H^n$. In case that $\Gamma$ has fixed points, $M$ is called a \textit{hyperbolic orbifold}.

One of the most important results that drove the study of hyperbolic geometry is the so-called \textit{Uniformization Theorem}, proved by Poincaré and Koebe in the beginning of the twentieth century. It states that any compact orientable topological surface $S_g$ with genus $g\geq 2$ admits a Riemannian metric locally isometric to $(\H^2,ds^2)$. In other words, there exists a torsion-free subgroup $\Gamma<\Isom^{+}(\H^2)$ such that $S_g$ is diffeomorphic to the quotient $\Gamma\bs\H^2$. The volume of a hyperbolic orbifold turns out to be a topological invariant. The first instance is given by the \textit{Gauss-Bonnet Theorem}. For a compact hyperbolic surface $S$ of genus $g$, it states that \[\area(S)=2\pi(2g-2).\]

For higher dimensions, it follows from results by Gromov and Thurston on simplicial volume. Our driving questions in this article look for relations of the volume with other geometric invariants on a hyperbolic manifold: the \textit{systole} and the \textit{kissing number}.

\begin{definition}
Let $M$ be a hyperbolic manifold of finite volume. The \emph{systole} of $M$, denoted by $\sys(M)$ is defined as
\[
\sys(M) := \inf_\alpha \{ \ell(\alpha) \mid \alpha \text{ is a non-contractible closed geodesic on  } M \}
\]
where $\ell(\alpha)$ denotes the length of $\alpha.$  The \emph{kissing number} of $M$, denoted by $\K(M)$, is given by
\[\K(M)=\#\{[\alpha]\neq 0\mid \alpha \text{ closed geodesic on $M$ with } \ell(\alpha)=\sys(M)\} \]
where $[\alpha]$ denotes the free homotopy class of $\alpha$.
\end{definition}

These are well-defined invariants since, in negative curvature, there are finitely many closed geodesics with the same length.

\subsection{Systole and volume of hyperbolic manifolds}

The first relation between systole and volume can be found in dimension $2$. For hyperbolic surfaces, we have the following.

\begin{proposition}\label{prop:sys area}
    Let $S$ be a compact hyperbolic surface. Then \begin{equation}\label{eq:sys vs area}
        \sys(S)\leq 2\log\left(\frac{\area(S)}{\pi}+2\right).
    \end{equation}
\end{proposition}

\begin{proof}
For any compact Riemannian surface $S$, the injectivity radius is equal to $\frac{\sys(S)}{2}$, i.e., for any point $p\in S$, the ball centered at $p$ of radius $\frac{\sys(S)}{2}$ is an embedded ball. Now, the area of a ball $\mathbb{B}_{R}$ of radius $R$ in $\mathbb{H}^{2}$ is equal to $2\pi(\cosh(R)-1)$. For any embedded ball $\mathbb{B}_{R}$ we have that $\area(S)\geq\area(\mathbb{B}_{R})$ and then, taking $R=\frac{\sys(S)}{2}$ we get
\[\frac{\area(S)}{2\pi}+1\geq\cosh\left(\frac{\sys(S)}{2}\right)\geq\frac{e^{\frac{\sys(S)}{2}}}{2},\]
and we obtain the desired formula by applying the logarithmic function.
\end{proof}

In dimension greater than two, we have the following estimate (\textit{see e.g.} \cite{BP22}).

\begin{proposition}
For any compact hyperbolic $n$-manifold $M$, there exists a constant $c_n>0$ depending only on $n$ such that
\begin{equation}\label{eq: sys vs vol higher dim c}
    \sys(M)\leq \frac{2}{n-1}\log(\vol(M))+c_n.
\end{equation}
\end{proposition}

\begin{proof}
We argue as in the proof of Proposition~\ref{prop:sys area}. The volume of an $n$-dimensional hyperbolic ball $\mathbb{B}_{R}^{n}$ of radius $R>0$ is given by the formula \cite[Ex. 3.4.6, page 79]{Ratcliffe}
\[\vol(\mathbb{B}_{R}^{n})=\vol(\mathbb{S}^{n-1})\int_{0}^{R}\sinh^{n-1}(t)dt,\]
where $\vol(\mathbb{S}^{n-1})$ denotes the volume of the Euclidean sphere of dimension $n-1$. Since $\vol(M)\geq\vol(\mathbb{B}_{\frac{\sys(M)}{2}})$, we need to bound from below the integral \(\int_{0}^{\frac{\sys(M)}{2}}\sinh^{n-1}(t)dt\).

Let us assume first that $\sys(M)\geq \ln2$. Since $\sinh(t)>0$ for $t>0$, and $\sinh(t)\geq \frac{1}{4}e^t$ for $t\geq\frac{\ln2}{2}$, we get in this case that
\[\int_{0}^{\frac{\sys(M)}{2}}\sinh^{n-1}(t)dt\geq \int_{\frac{\ln(2)}{2}}^{\frac{\sys(M)}{2}}\sinh^{n-1}(t)dt\geq \int_{\frac{\ln(2)}{2}}^{\frac{\sys(M)}{2}}\frac{1}{4^{n-1}}e^{t(n-1)}dt,
\]
and then
\[\vol(M)\geq\frac{\vol(\mathbb{S}^{n-1})}{4^{n-1}}\int_{\frac{\ln(2)}{2}}^{\frac{\sys(M)}{2}}e^{t(n-1)}dt.\]

From this, we obtain that
\begin{align*}
 \sys(M)& \leq  \frac{2}{n-1}\log\left(\frac{\vol(M)4^{n-1}(n-1)}{\vol(\mathbb{S}^{n-1})}+e^\frac{(n-1)\log(2)}{2}\right)\\
 &= \frac{2}{n-1}\log(\vol(M))+\log\left(\frac{4^{n-1}(n-1)}{\vol(\mathbb{S}^{n-1})}+\frac{e^\frac{(n-1)\log(2)}{2}}{\vol(M)}\right).
\end{align*}

Due to the Kazhdan-Margulis Theorem, there is a constant $v_n$ depending only on $n$ such that $\vol(M)\geq v_n$. Therefore, equation~\eqref{eq: sys vs vol higher dim c} follows with \[c_n=\log\left(\frac{4^{n-1}(n-1)}{\vol(\mathbb{S}^{n-1})}+\frac{e^\frac{(n-1)\log(2)}{2}}{v_n}\right).\]

If $\sys(M)<\ln(2)$, we can increase the constant $c_n$ by $\ln(2)$ if necessary, so equation~\eqref{eq: sys vs vol higher dim c} still holds.
\end{proof}

When the hyperbolic manifold $M$ is non-compact, the argument of the injectivity radius does not apply. Work by M. Gendulphe \cite{Gen15} implies that for any non-compact hyperbolic manifold $M$ with finite volume, we have
\begin{equation}\label{eq:gendulphe}
  \sys(M)\leq 2\log(\vol(M))+d_{1},
\end{equation}
where $d_{1}>0$ is an explicit constant depending on the dimension of $M$. It is worth to note that for $n=3$, G. S. Lakeland and C. J. Leininger proved that
\begin{equation}
  \sys(M)\leq \frac{4}{3}\log(\vol(M))+d_{2},
\end{equation}
for some uniform constant $d_{2}>0$ \cite{LL14}.

Using Teichm\"uller theory, we can deform the metric on a hyperbolic surface $S$, keeping it hyperbolic with the same area, to obtain a new metric having arbitrarily small closed curves, and hence arbitrary small systole. The existence of compact hyperbolic 3-manifolds with arbitrarily short closed geodesics follows from Thurston's hyperbolic Dehn surgery theorem \cite[Thm. 5.8.2]{Thur77}. In 2006, I. Agol constructed compact hyperbolic $4$-manifolds with arbitrarily short systoles \cite{Agol06}, and M. Belolipetsky and S. Thomson in 2011 adapted Agol's method to any dimension \cite{BT11} (see also \cite[Remark on Section 9]{BHW11}). More recently, other constructions of closed hyperbolic manifolds with arbitrarily short systole have been obtained, see \cite[Thm. 1 and Thm. 4]{D25}.

\subsection{Kissing number and volume of hyperbolic manifolds}

In 2013, Hugo Parlier proved that the kissing number of a hyperbolic surface is bounded from above in terms of the area of the surface, which can be expressed in terms of the genus by the
Gauss-Bonnet Theorem.

\begin{theorem}[\!\!{\cite[Thm. 1.3]{Parlier13}}]\label{th:Par}
    There exists a constant $U> 0$ such that any closed hyperbolic surface $S_g$ of genus $g$ satisfies
    \begin{equation}\label{eq:kiss vs gen}
      \K(S_g)\leq U\cdot \frac{g^2}{\log g}.
    \end{equation}
\end{theorem}

For non-compact hyperbolic orbifolds of finite area, the corresponding result was obtained in 2016 by H. Parlier and F. Fanoni.

\begin{theorem}[\!{\cite[Thm 1.1-1.2]{FP16}}]\label{th:ParFan}
Any hyperbolic surface $S$ with genus $g\geq 1$ and $n$ cusps, such that $3g-3+n>0$, satisfies
\begin{align}
    \K(S)&\leq C(g+n)\frac{g}{\log(g+1)}, \\
    \sys(S)&\leq 2\log(g)+K
\end{align}
for some universal constants $C>0$ and $K < 8$.
\end{theorem}

For $n>2$, Theorem \ref{th:Par} was generalized by Maxime Bourque and Bram Petri in 2022.

\begin{theorem}[\!\!{\cite[Cor. 1.2]{BP22}}]\label{th:BP}
For every closed hyperbolic $n$-manifold $M$ it holds that
\begin{equation}\label{eq:kiss vs vol n high}
  \K(M)\leq A_{n}\frac{\vol(M)^{2}}{\log(1+\vol(M))}
\end{equation}
for certain constant $A_{n}>0$ depending only on $n$.
\end{theorem}

So far, it is not yet known whether there is a version of Theorem \ref{th:BP} for non-compact hyperbolic manifolds of finite volume.

\subsection{Two guiding questions}\label{sec:two guiding questions}

We have seen that there exist hyperbolic manifolds with arbitrarily short closed geodesics. In general, it follows from a classical result of Anosov that a generic Riemannian manifold has at most one systole \cite{Ano83}. So, it is natural to ask to what extent inequalities \eqref{eq:sys vs area}–\eqref{eq:kiss vs vol n high} are optimal. One way to formalize this is by bounding from the other side as follows:

\begin{question}\label{q:q1}
    Is there a constant $C_1>0$ such that \[\limsup_{\vol(M)\to\infty,
    \dim(M)=n}\left(\frac{\sys(M)}{\log (\vol(M))}\right)\geq C_1?\]
\end{question}

\begin{question}\label{q:q2}
       Is there a constant $C_2>1$ such that \[\limsup_{\vol(M)\to\infty,
    \dim(M)=n}\left(\frac{\log(\K(M))}{\log (\vol(M))}\right)\geq C_2?\]
\end{question}

Suppose that there is a sequence of finite volume hyperbolic manifolds $M_i$ whose volumes tend to $\infty$, such that
\begin{equation}\label{eq: M with large kiss}
    \K(M_i)\geq\frac{\vol(M_i)^{1+\epsilon}}{\log(\vol(M_i))}
\end{equation}
for some $\epsilon>0$. Then, \cite[Thm 1.1]{BP22} implies that $\sys(M_i)$ is essentially bounded from below by $\frac{2\epsilon}{n-1}\log(\vol(M_i))$ (see the introduction of \cite{DFM24} for more details).

Then, a solution to Question~\ref{q:q2} gives an answer to Question~\ref{q:q1}. However, little is known about Question~\ref{q:q2}. We will see that an answer to Question~\ref{q:q2} can be obtained from hyperbolic manifolds that are \textit{arithmetic} (see Theorem~\ref{th:mainth} and Section~\ref{sec:kiss high dim}). In the remainder of this article, we will do so for hyperbolic surfaces, and it will serve us as an introduction for the theory of arithmetic groups and arithmetic manifolds.
\section{Hyperbolic surfaces and the trace-length relation}

In this section, we recall some facts about hyperbolic surfaces and their connection to the real Lie group $\PSL_2(\R)$. Most of these facts are well-known and can be found, for instance, in \cite{Katok92}.

\subsection{Constructing hyperbolic surfaces}

Consider the matrix group $\SL_2(\R)$ given by
\[\SL_2(\R)=\Biggl\{\begin{pmatrix}
    a & b\\
    c & d
\end{pmatrix}; a,b,c,d\in\R \text{ and } ad-bc=1\Biggr\}.\]

Any matrix $\gamma=\begin{pmatrix}
    a& b\\
    c&d
\end{pmatrix}\in\SL_2(\R)$ defines a M\"obius transformation $T_\gamma:\H^2\rightarrow\H^2$ given by
\begin{align}
T_\gamma(z)=\frac{az + b}{cz + d}\label{eq:action PSL on H2}
\end{align}
and $T_\gamma$ preserves the metric and orientation of $\H^2$. That is, $T_\gamma\in\Isom^{+}(\H^2)$. We will use $\gamma$ or $T_\gamma$ interchangeably. Observe that $T_\gamma=T_{-\gamma}$, and then we get a well-defined map from $\PSL_2(\R)=\SL_2(\R)/\{\pm I_2\}$ to $\Isom^{+}(\H^2)$ given by
\begin{align*}
\psi:\PSL_2(\R)&\rightarrow\Isom^{+}(\H^2)\\
\pm\gamma&\mapsto T_\gamma
\end{align*}
It turns out that $\psi$ defines a group isomorphism between $\PSL_2(\R)$ and $\Isom^{+}(\H^2)$ \cite[Section 1.3]{Katok92}. The group $\SL_2(\R)$, with the topology as a closed subset of $\R^4$, is a locally compact topological group. With respect to the quotient topology induced by $\SL_2(\R)$, the group $\PSL_2(\R)$ inherits the structure of a locally compact topological group. In this way, the isomorphism $\psi$ allows us to look at $\Isom^{+}(\H^2)$ as a locally compact topological group with respect to the topology of $\PSL_2(\R)$. With this topology, we can understand when a group $\Gamma<\Isom^{+}(\H^2)$ acts proper discontinuously without fixed points on $\H^2$.

\begin{proposition}
   A subgroup $\Gamma<\Isom^{+}(\H^2)$ acts proper discontinuously without fixed points on $\H^2$ if and only if $\Gamma$ is a discrete and torsion-free subgroup.
\end{proposition}

Therefore, the problem of constructing hyperbolic surfaces is translated to the problem of constructing discrete and torsion-free subgroups on $\PSL_2(\R)\cong  \Isom^{+}(\H^2)$.
\begin{definition}
    A \emph{Fuchsian group} is a discrete subgroup of $\PSL_2(\R)$.
\end{definition}

 The theory of arithmetic groups provides a way to construct Fuchsian groups, the so-called arithmetic Fuchsian groups. We will see an instance of this in Section~\ref{sec:Fuchsian groups}.

\subsection{Trace-length relation}\label{sec:trace-length}

Equation~\ref{eq:action PSL on H2} also defines an action of $\PSL_2(\R)$ on the boundary $\partial\H^2=\R\cup\{\infty\}$, so $\PSL_2(\R)$ acts on the compactification $\overline{\H^2}=\H^2\cup\partial\H^2$, which is homeomorphic to the $2$-dimensional disc $\mathbb{D}^2$. An element can be characterized in $\PSL_2(\R)$ according to its fixed points on $\overline{\H^2}$ as follows. An element \( \gamma \in \mathrm{PSL}_2(\mathbb{R}) \) is called:

\begin{itemize}
    \item \textit{Elliptic}, if \( \gamma \) fixes a point in \( \mathbb{H}^2 \).
    \item \textit{Parabolic}, if \( \gamma \) has no fixed points in \( \mathbb{H}^2 \) and has a unique fixed point on \( \partial \mathbb{H}^2 \).
    \item \textit{Hyperbolic}, if \( \gamma \) has no fixed points in \( \mathbb{H}^2 \) and has two fixed points on \( \partial \mathbb{H}^2 \).
\end{itemize}

The following is a basic fact that can be obtained directly from the definition above.

\begin{proposition}
Let \( \beta \in \mathrm{PSL}_2(\mathbb{R}) \). Then
\begin{itemize}
    \item \( \beta \) is elliptic if and only if \(|\mathrm{tr} \beta| < 2 \).
    \item \( \beta \) is parabolic if and only if \(  |\mathrm{tr} \beta| = 2 \).
    \item \( \beta \) is hyperbolic if and only if \(  |\mathrm{tr} \beta| > 2 \).
\end{itemize}
Here \( \mathrm{tr}(\beta) \) denotes the trace of a representative of \( \beta \) in $\SL_2(\R)$.
\end{proposition}

\begin{definition}
Let \( \gamma \in \mathrm{PSL}_2(\mathbb{R}) \) be hyperbolic. The \emph{translation length} of \( \gamma \) is defined as:
\[
\ell_\gamma = \inf \left\{ d_{\mathbb{H}^2}(z, \gamma(z)); z \in \mathbb{H}^2 \right\}.
\]
\end{definition}

Let $S=\Gamma\bs \mathbb{H}^2$ be a hyperbolic surface. The set of conjugacy classes of hyperbolic elements in \( \Gamma \) is in bijection with the set of free homotopy classes of closed geodesics in $S$. For any hyperbolic element $\gamma\in\Gamma$, the two fixed points $z_1,z_2\in\partial\H^2$ of $\gamma$ determine a geodesic $\tilde\alpha_\gamma$ in $\H^2$ invariant by the action of $\gamma$. The value $\ell_\gamma$ is equal to $d_{\mathbb{H}^2}(z, \gamma(z))$ for any $z$ on $\tilde\alpha_\gamma$.  Moreover, $\tilde\alpha_\gamma$ projects on a closed geodesic $\alpha_\gamma$ of $\Gamma\bs\H^2$ of length $\ell(\alpha_\gamma)=\ell_\gamma$. Conversely, a closed geodesic $\alpha$ corresponds to a deck transformation $\gamma\in\Gamma$. Also, $\alpha$ lifts to a geodesic $\tilde\alpha$ on $\H^2$ with two distinct endpoints in $\partial\H^2$. Since $\gamma$ corresponds to $\alpha$ there is a $g\in\Gamma$ such that $g\gamma g^{-1}$ leaves $\tilde\alpha$ invariant. It means that the end points of $\tilde\alpha$ are fixed points of $g\gamma g^{-1}$, and then $g\gamma g^{-1}$ is hyperbolic. This implies that $\gamma$ is hyperbolic.

\begin{example}\label{ex:translation length canonical matrix}
Consider a real number $\lambda>1$ and let $\beta\in\SL_2(\R)$ be the matrix
\begin{equation}\label{canonical matrix}
\beta=\begin{pmatrix}
        \lambda & 0 \\
        0 & \lambda^{-1}
    \end{pmatrix}.
\end{equation}
Then, $\beta$ is hyperbolic, and has fixed points $0,\infty$ in $\partial\H^2$. Moreover, $\beta$ leaves invariant the geodesic $\tilde\alpha(t)=e^{t}i$ in $\H^2$. Taking $z=i,$ the translation length of $\beta$ is given by \[\ell_\beta=d_{\H^2}(i,\beta(i))=d_{\H^2}(i,\lambda^2i)=2\log(\lambda).\]
\end{example}

The previous example shows a connection between the eigenvalues of $\beta$ and $\ell_\beta$ for $\beta$ in the form \eqref{canonical matrix}. In general, the translation length can be read from the trace of the corresponding element in $\PSL_2(\mathbb{R})$ as the following proposition shows. We include a proof as it is short, and similar arguments can be used in higher dimensions.
 \begin{proposition}[Trace-length relation]\label{prop:translation length}
 Let $\gamma \in \mathrm{PSL}_2(\mathbb{R})$ be hyperbolic. Then
\[
\cosh\left( \frac{\ell_\gamma}{2} \right) = \frac{|\tr(\gamma)|}{2}.
\]
In particular, \(
\ell_\gamma \geq 2 \log\left( |\tr(\gamma)| - 1\right) > 0.
\)
\end{proposition}

\begin{proof}
    Since $|\tr(\gamma)|>2$ and $\det(\gamma)=1$ we have that $\gamma$ is conjugated in $\PSL_2(\R)$ to a matrix in the form \eqref{canonical matrix}. That is, there is $\beta\in\PSL_2(\R)$ and $\lambda>1$ such that
    \begin{equation*}
\beta\gamma\beta^{-1}=\begin{pmatrix}
        \lambda & 0 \\
        0 & \lambda^{-1}
    \end{pmatrix}.
\end{equation*}

Since $T_\beta$ is an isometry of $\H^2$, we have that $\ell_\gamma=\ell_{T_{\beta}A_{\gamma}T_{\beta}^{-1}}$. By Example~\ref{ex:translation length canonical matrix}
we get then that $\ell_\gamma=2\log(\lambda)$. Therefore, \[\cosh\left(\frac{\ell_\gamma}{2}\right)=\frac{e^{\frac{\ell_\gamma}{2}}+e^{\frac{-\ell_\gamma}{2}}}{2}=\frac{\lambda+\lambda^{-1}}{2}=\frac{|\tr(\gamma)|}{2}.\]

Now, since $\displaystyle\frac{e^{\frac{\ell_\gamma}{2}}+1}{2}>\cosh\left(\frac{\ell_\gamma}{2}\right)$ we obtain that \(\displaystyle
\ell_\gamma \geq 2 \log\left( |\tr(\gamma)| - 1\right) > 0
\)
\end{proof}

\section{Arithmetic Fuchsian groups}\label{sec:Fuchsian groups}

The goal of this section is to introduce the class of Fuchsian groups that are \textit{arithmetic}.  These groups are generalizations of $\SL_2(\Z)$, and the general construction involves groups of units of quaternion algebras over number fields with certain conditions. To keep the presentation simpler, we have restricted to the class of arithmetic Fuchsian groups defined over $\Q$. We have chosen an exposition that naturally leads to quaternion algebras from the group $\SL_2(\Z)$ itself.

\subsection{The modular group}\label{sec: Modular group}
We can construct discrete subgroups of \( \mathrm{PSL}_2(\mathbb{R}) \), so Fuchsian groups, by projecting discrete subgroups of \( \mathrm{SL}_2(\mathbb{R}) \). The simplest example is the so-called \textit{Modular group}
\[\Gamma = \mathrm{SL}_2(\mathbb{Z}) = \left\{ \begin{pmatrix} a & b \\ c & d \end{pmatrix} : a, b, c, d \in \mathbb{Z}, \ ad - bc = 1 \right\}.\]

The modular group and this representation is widely known and used, for instance in research on dynamical systems and geometry. There is another representation of $\SL_2(\Z)$ that is less known among researchers in these areas, but well-known in algebraic number theory. That is, $\SL_2(\Z)$ as the group of integer points of the group of units of a quaternion algebra over $\Q$. This is the point of view that will help us to construct the class of arithmetic Fuchsian groups mentioned above. Let us rewrite $\mathrm{SL}_2(\mathbb{Z})$ following three steps:

\subsection*{Step 1}

Take \( A = M_2(\mathbb{Q}) \), the algebra of \( 2 \times 2 \) matrices over \( \mathbb{Q} \).

\subsection*{Step 2}
Consider the group \( A^1 \subset A \) of rational matrices with determinant 1, that is
\[
A^1 = \SL_2(\Q)=\left\{ \begin{pmatrix} a & b \\ c & d \end{pmatrix} \in A : ad - bc = 1 \right\}.
\]

\subsection*{Step 3}
Inside \( A^1 \), we select matrices with integer coefficients:
\[
A^1(\mathbb{Z}) = \left\{ \begin{pmatrix} a & b \\ c & d \end{pmatrix} \in A : ad - bc = 1,\ a, b, c, d \in \mathbb{Z} \right\}.
\]

Now, to introduce arithmetic groups over $\Q$, let us rewrite the algebra \( A = M_2(\mathbb{Q}) \) as follows.
\begin{align*}
    A = M_2(\mathbb{Q})& = \left\{x= \begin{pmatrix} x_0 + x_1  & x_2 + x_3 \\ x_2 - x_3 & x_0 - x_1 \end{pmatrix} \middle| x_i \in \mathbb{Q} \right\}\\
    &=\left\{x=x_0\begin{pmatrix} 1  & 0 \\ 0 & 1 \end{pmatrix}+x_1\begin{pmatrix} 1  & 0 \\ 0 & -1 \end{pmatrix}+x_2\begin{pmatrix} 0  & 1 \\ 1 & 0 \end{pmatrix}+x_3\begin{pmatrix} 0  & 1 \\ -1 & 0 \end{pmatrix}, x_i\in\Q\right\}.
\end{align*}

Note that, under this representation $x=\begin{pmatrix} x_0 + x_1  & x_2 + x_3 \\ x_2 - x_3 & x_0 - x_1 \end{pmatrix}$ lies in $\SL_2(\Z)$ if and only if $2x_0,2x_1,2x_2,2x_3\in\Z$ and $x_0^2-x_1^2-x_2^2+x_3^2=1$. Now, denote by $1,i,j,k$ the matrices
\[1= \begin{pmatrix} 1 & 0 \\ 0 & 1 \end{pmatrix},\quad i = \begin{pmatrix} 1 & 0 \\ 0 & -1 \end{pmatrix},\quad j = \begin{pmatrix} 0 & 1 \\ 1 & 0 \end{pmatrix}, \quad k = \begin{pmatrix} 0 & 1 \\ -1 & 0 \end{pmatrix}.\]
Then, in $A$ we have the relations \[
i^2 = 1, \quad j^2 = 1, \quad ij = -ji = k\]
\begin{align*}
    \det(x) &= x_0^2 -x_1^2 - x_2^2 + x_3^2.\\
    \tr(x)&=2x_0.
\end{align*}

Thus, steps (1)-(3) can be rewritten as follows:
\begin{align*}
 A &= \left\{ x = x_0 + x_1 i + x_2 j + x_3 k \middle| \hspace{1mm}x_i \in \mathbb{Q},\hspace{1mm} i^2 = 1,\hspace{1mm} j^2 = 1,\hspace{1mm} ij = -ji = k\right\},\\
A^1& = \left\{ x \in A \mid x_0^2 - x_1^2 - x_2^2 + x_3^2 = 1, x_i\in\Q \right\},\\
 A^1(\Z)& = \left\{ x \in A \mid x_0^2 - x_1^2 - x_2^2 + x_3^2 = 1, 2x_i\in\Z \right\}.
\end{align*}

Let $K$ be a field. A \textit{quaternion algebra over $K$} is a $4$-dimensional non-commutative $K$-algebra which has a basis $\{1,i,j,k\}$ such that
\[
A=A_{a,b} = \left( \frac{a,b}{K} \right) = \left\{ x = x_0 + x_1 i + x_2 j + x_3 k \middle| x_i \in K,\ i^2 = a, j^2 = b, ij = -ji = k \right\}
\]
for some $a,b\in K^{\times}$. We have just seen $M_2(\Q)$ is a quaternion algebra over $\Q$. Exploiting this reinterpretation of $\SL_2(\Z)$ will allow us to construct more examples of Fuchsian groups.

\subsection{Arithmetic Fuchsian groups defined over $\Q$.}\label{sec:arith. Fuch. Groups.}

Let $a,b \in \mathbb{Q}^\times$, $a>0$ and consider the quaternion algebra $A_{a,b}$ over $\mathbb{Q}$ defined by
\[
A_{a,b} = \left( \frac{a,b}{\mathbb{Q}} \right) = \left\{ x = x_0 + x_1 i + x_2 j + x_3 k \middle| x_i \in \mathbb{Q},\ i^2 = a, j^2 = b, ij = -ji = k \right\}.
\]

Consider the map  $\varphi : A_{a,b} \rightarrow M_2(\mathbb{R})$ given by
\[\varphi(x)=\begin{pmatrix}
  x_0 + x_1 \sqrt{a} & x_2 + x_3 \sqrt{a} \\
  b(x_2 - x_3 \sqrt{a}) & x_0 - x_1 \sqrt{a}
\end{pmatrix}.
\]
Straightforward computations show that $\varphi$ is an injective $\Q$-algebras homomorphism, and that
\begin{align*}
\det(\varphi(x))& = x_0^2 - a x_1^2 - b x_2^2 + ab x_3^2\\
\tr(\varphi(x)) &= 2x_0.
\end{align*}

Following the previous section\footnote{In general,  for $A_{a,b}(\Z)$ it is easier to take the coefficients lying in $\Z$ instead of in $\frac{1}{2}\Z$ as in $\SL_2(\Z)$. This difference is compensated by the commensurability condition.}, let
\begin{align*}
A^1_{a,b}&= \left\{ x \in A_{a,b} \mid x_0^2 - a x_1^2 - b x_2^2 + ab x_3^2 = 1 \right\};\\
A^1_{a,b}(\Z)&= \left\{ x \in A_{a,b} \mid x_0^2 - a x_1^2 - b x_2^2 + ab x_3^2 = 1, x_i\in\Z \right\}.
\end{align*}

Therefore, for $a>0$, $\varphi$ restricts to an injective group homomorphism between $A_{a,b}^1$ and $\SL_2(\R)$. So, $\varphi(A^1_{a,b}(\mathbb{Z}))$ is a subgroup of $\mathrm{SL}_2(\mathbb{R})$, in a similar way that $\SL_2(\Z)$ is. We have arrived to the main definition of this section. Before, we say that two subgroups $\Gamma_1,\Gamma_2<\SL_2(\R)$ are \textit{commensurable} if there is $g\in\SL_2(\R)$ such that $\Gamma_1\cap g\Gamma_2g^{-1}$ has finite index in both $\Gamma_1$ and $g\Gamma_2g^{-1}$.

\begin{definition}
    A group $\Gamma < \mathrm{SL}_2(\mathbb{R})$ is an \emph{arithmetic subgroup defined over $\Q$}, if $\Gamma$ is commensurable with $\varphi(A^1_{a,b}(\mathbb{Z}))$ for some $a,b \in\Q^{\times}, a>0.$
\end{definition}

It is not difficult to see that $\varphi(A^1_{a,b}(\mathbb{Z}))$ is a discrete subgroup of $\SL_2(\R)$. However, a fundamental theorem, which is a wide generalization due to Borel and Harish-Chandra (see \cite{BHC62}), states that $\varphi(A^1_{a,b}(\mathbb{Z}))$ has finite co-area in $\SL_2(\R)$. The compactness of $\varphi(A^1_{a,b}(\mathbb{Z}))\bs\H^2$ depends entirely on the arithmetic properties of $A_{a,b}$ (see \cite[Proposition~6.2.4]{Mor15}).

\begin{theorem}\label{pro:Godement}
Let $a,b\in\Q^\times, a>0$ and $A_{a,b}$ be as in Section~\ref{sec:arith. Fuch. Groups.}. Then
\begin{enumerate}
    \item $\varphi(A^1_{a,b}(\mathbb{Z}))$ is a discrete subgroup of $\SL_2(\R)$ of finite coarea.

    \item $\varphi(A^1_{a,b}(\mathbb{Z}))\bs\H^2$ is compact if and only if $x_0^2-ax_1^2-bx_2^2+abx_3^2=0$ has no solutions in $\Q^4-\{0\}.$
\end{enumerate}
\end{theorem}
Fuchsian groups that are arithmetic are called \textit{arithmetic Fuchsian groups.}
\begin{example}\label{ex:cocompact examples}
    Let $a=p$ be a prime $p\equiv 3\mod 4$ and $b=-1.$ In this case, the equation $x_0^2-px_1^2+x_2^2-px_3^2=0$ has no non-trivial rational solutions. Indeed, if that happens, we can suppose that $x_i\in\Z$ are relatively prime, and then $x_0^2+x_2^2=p(x_1^2+x_3^2).$ This contradicts the Sum of Squares Theorem, which states that a natural number, which is a sum of two squares, necessarily has an even number of prime factors congruent to $3\mod 4$.
 \end{example}

 From Theorem~\ref{pro:Godement}, we obtain that the group $\Gamma=\varphi(A^1_{p,-1}(\mathbb{Z}))$ is a cocompact arithmetic Fuchsian group for the algebra $A_{p,-1}=\left(\frac{p,-1}{\Q}\right)$.

\section{Hyperbolic surfaces with large systole and kissing number.}\label{sec:hyp surf larg sys and kiss}

Now that we have introduced the class of arithmetic Fuchsian groups defined over $\Q$, we are able to provide answers to Question~\ref{q:q1} and Question~\ref{q:q2}. It is important to mention that the definition of the kissing number appears for the first time in works by S. Schaller, and also in the first contributions to its growth for arithmetic surfaces \cite{Sch96-1, Sch96-2, Sch96-3}. The results that we will present in this section are due to P. Buser - P. Sarnak, and S. Schaller. We will mainly follow the exposition in  \cite{Sch96-1} together with ideas of \cite[Sec. 4]{BS94} to prove the following result.

\begin{theorem}\label{th:mainth}
There exists a sequence of finite area (compact and non-compact) hyperbolic surfaces $S_i$ with $\text{area}(S_i) \to \infty$ such that:
\begin{enumerate}
    \item $\mathrm{Kiss}(S_i) \geq D \cdot (\area(S_i))^{\frac{4}{3} - \varepsilon}$
    \item $\sys(S_i) \geq \frac{4}{3} \log(\area(S_i)) - c$
\end{enumerate}
 For some constants $\varepsilon > 0,D > 0, c>0$.
\end{theorem}

In the compact case, Theorem \ref{th:mainth} implies that

\begin{corollary} Let $S_g$ denote a compact hyperbolic surface of genus $g\geq 2.$ Then
\begin{align*}
    \limsup_{g\to\infty}\left(\frac{\sys(S_g)}{\log g}\right)&\geq \frac{4}{3};\\
    \limsup_{g\to\infty}\left(\frac{\log(\K(S_g))}{\log g}\right)&\geq \frac{4}{3}.
\end{align*}
\end{corollary}

Let us explain the strategy to prove Theorem \ref{th:mainth}. Let $S=\Gamma\bs\H^2$ be a hyperbolic surface. First, by the discussion of Section~\ref{sec:trace-length} in order to have a large number of systoles, we will need to ensure that there is a large number of hyperbolic non-pairwise conjugated elements in $\Gamma$, which have the same trace. This will certainly produce a large number of closed geodesics in  $\Gamma\bs\mathbb{H}^2$  with the same length. However, they may not be systoles. So, additionally we need to guarantee that such a large number of elements in $\Gamma$ with the same trace also produce systoles in $S$. How to do so is not obvious at all.

The construction that will be presented in the following uses the arithmetic Fuchsian groups introduced in Section~\ref{sec:Fuchsian groups}, and produces the best possible answers to Question~\ref{q:q1} and Question~\ref{q:q2} that have been given so far. Following the goal of this expository article, that is, to give a down-to-Earth introduction to arithmetic groups, we will first prove Theorem \ref{th:mainth} in the noncompact case. Here, all the geometric arguments hold also in the compact setting, but the algebraic tools involved will make use of the modular group $\SL_2(\Z)$. So, we will have the advantage to introduce all the main ideas in an easier context, and the reader will see that passing to the compact case is only a matter of technical work involving arithmetic groups. In the same spirit of making the exposition easier, for the compact case, we will only make use of the family of surfaces introduced in Example~\ref{ex:cocompact examples}.

For both the non-compact and compact cases, we will use the following:
\begin{definition}
Let $\Gamma$ be a Fuchsian group. An element $\gamma \in \Gamma$ is \emph{primitive} if $\gamma = \beta^k$, $\beta \in \Gamma$ implies $k = 1$ and $\gamma = \beta$.
\end{definition}

\subsection{The non-compact case.}

Let us consider the modular group $\Gamma = \mathrm{SL}_2(\mathbb{Z})$ with the associated modular surface $S=\Gamma\bs\H^{2}$. Although $S$ is not a smooth surface, as $\Gamma$ has elliptic elements, $S$ will serve as a base space from where the surfaces $S_i$ will be constructed. First, let us see that there are a large number of non-pairwise conjugated matrices in $\Gamma$  with the same trace. For any $\gamma\in\SL_2(\Z)$ we denote by $[\gamma]$ the conjugacy class of $\gamma$ in $\SL_2(\Z)$.

\begin{theorem}[see \cite{CCC79}]\label{th: Siegel}
Let
\[
\mu_0(t) = \#\left\{ [\gamma] \in \mathrm{SL}_2(\mathbb{Z}) \ \text{primitive, } |\tr(\gamma)| = t \right\} .
\]
For any $\epsilon>0$ there exist $T(\epsilon) > 0$ such that
\[
\mu_0(t) > t^ {1-\epsilon},\quad \text{for all } t > T(\epsilon).
\]
\end{theorem}

Now we will construct surfaces that will serve to our purpose. For each $N\in\N$, define $\Gamma(N)$ as the \textit{principal congruence subgroup of level N} of $\SL_2(\Z)$, given by
\[
\Gamma(N) := \ker\left( \mathrm{SL}_2(\mathbb{Z}) \xrightarrow{\pi_N} \mathrm{SL}_2(\mathbb{Z}/N\mathbb{Z}) \right),\]
where $\pi_N$ denotes the reduction map $\hspace{-1mm}\mod N$. Then \(
\Gamma(N)\) is a normal subgroup of \(\mathrm{SL}_2(\mathbb{Z})\). The quotient $S_N =\Gamma(N)\bs \mathbb{H}^2$ is a finite-sheeted covering space of $S$, and it is called  \textit{principal congruence covering of level $N$} of $S$. The next result shows that $S_N$ is in fact a hyperbolic surface and tells us how to detect the systoles in $S_N$.

\begin{lemma}\label{prop:SN}
Let $N \geq 2$ and $\Gamma(N)$ be the principal congruence subgroup of level $N$ of $\Gamma=\SL_2(\Z)$.

    \begin{enumerate}
        \item[(a)] For $\gamma \in \Gamma(N)$, if $|\tr(\gamma)|\neq 2$  then $|\tr(\gamma)| \geq N^2-2$.
        \item[(b)] For $\beta \in \Gamma$ with $|\tr(\beta)|= N$, then $-\beta^2 \in \Gamma(N)$ and $-\beta^2$ represents  a systole of $S_N$.
        \item[(c)] $S_N$ is a non-compact hyperbolic surface with $\area(S_N)< \area(S)\cdot N^3$.
    \end{enumerate}
\end{lemma}

\begin{proof}
 \begin{enumerate}
     \item[(a)] Let $\gamma \in \Gamma(N)$ such that
\(
\gamma =
\begin{pmatrix}
1 + x & y \\
z & 1+w
\end{pmatrix}, \quad x, y, z, w \in N\mathbb{Z}\).
The equation $\det(\gamma)=1$ is equivalent to the equation \[x+w+xw-zy=0\]
Therefore $x+w\in N^2\cdot\Z$. Since $\gamma$ is not parabolic by assumption, then $x+w\neq 0$ so $|x+w|\geq N^2$. It follows by the triangle inequality that $|\tr(\gamma)|=|2+x+w|\geq N^2-2$.

\item[(b)] By the Cayley-Hamilton Theorem, $\beta$ satisfies the equation
\[
X^2 - \tr(\beta)X + 1 = 0
\]
thus
\[
-\beta^2= -\tr(\beta)\beta + I \equiv I \mod N.
\]
We also get that $\tr(-\beta^2) = -N^2 + 2$, and then $|\tr(-\beta^2)| = N^2 - 2$. From $(a)$, $-\beta^{2}$ defines a systole of $S_N$.

\item[(c)] From (a) we get that $\Gamma(N)$ has no elliptic elements, so $S_N$ is smooth, and then $S_N$ is a hyperbolic surface. Since $[\Gamma:\Gamma(N)]$ is finite, then
\begin{align*}
    \area(S_N)&=\area(S)\cdot[\Gamma:\Gamma(N)]\\
    &\leq\area(S)|\SL_2(\Z/N\Z)|.
\end{align*}

The result follows from the fact that \(\left| \mathrm{SL}_2(\mathbb{Z}/N\mathbb{Z}) \right|=N^3\cdot\prod_{p|N}(1-\frac{1}{p^2}) < N^3\), where the product is taken over the prime integers dividing $N$.
\qedhere \end{enumerate}
\end{proof}

We can now prove Theorem \ref{th:mainth} in the non-compact case.
\begin{theorem}\label{th:large kiss and systole nc}
    For $N$ sufficiently large, $S_N$ satisfies
\begin{align*}
\K(S_N)& \geq \frac{\operatorname{area}(S_N)^{\frac{4}{3}-\epsilon}}{2v_0^2},\\
\sys(S_N)& \geq \frac{4}{3}\log(\operatorname{area}(S_N)) - \left(2\log(2) + \frac{4}{3}\log(v_0)\right),
\end{align*}
for any $\epsilon>0$, where $v_0=\area(S)$.
\end{theorem}

\begin{proof}
For any $\varepsilon > 0$ let $\varepsilon'=3\varepsilon$, $T(\varepsilon')$ given by Theorem~\ref{th: Siegel}, and take $N> T(\varepsilon')$. The group $\overline\Gamma = \mathrm{PSL}_2(\mathbb{Z})$ has $\mu_0(N)>N^{1-\epsilon'}$ conjugacy classes of primitive elements of trace $\pm N$, namely $\beta_1, \beta_2, \dots, \beta_{\mu_0(N)}$. By Lemma \ref{prop:SN}, the elements $-\beta_1^2,-\beta_2^2, \dots, -\beta_{\mu_0(N)}^2$ represent distinct systoles in $S_N$.

Denote by $X(N)$ the set of elements $\gamma\in\Gamma(N)$ with $|\tr(\gamma)|=N^2-2$. So, the elements \(-\beta_1^2,-\beta_2^2,\ldots, -\beta_{\mu_0(N)}^2\) belong to \(X(N)\). Since $\Gamma(N)$ is normal in $\Gamma$, the group $\Gamma/\Gamma(N)$ acts by conjugation on $X(N)$. Denote by $\overline{\gamma}$ the class of $\gamma$ in $\Gamma/\Gamma(N)$, and suppose that $\gamma(-\beta_i^2) \gamma^{-1} = - \beta_i^2$. The equation
\[
\gamma(\beta_i^2-\tr(\beta_i)\beta_i+I)\gamma^{-1} = 0
\]
implies that $\gamma \text{ commutes with } \beta_i$, and therefore the fixed points of $\gamma$ and $\beta_i$ in $\partial\mathbb{H}^2$ are equal. Since $\beta_i$ is primitive we get that $\gamma \in \langle \beta_i \rangle$. So, $\overline \gamma$ has order $2$ in $\Gamma/\Gamma(N)$. This implies that each geodesic $-\beta_i^2$ in $S_N$ produces
\(
\frac{|\Gamma/\Gamma(N)|}{2}
\)
distinct closed geodesics in $S_N$ of the same length.

Note that $\gamma \beta_i^2 \gamma^{-1}\neq \beta_j^2$ for $i\neq j$. Indeed, if  $\gamma \beta_i^2 \gamma^{-1}= \beta_j^2$ for some $i\neq j$, the equations
\begin{align*}
\gamma(\beta_i^2-N\beta_i+I)\gamma^{-1} &= 0\\
\beta_j^2-N\beta_j+I &= 0
\end{align*}
imply that $\gamma\beta_i\gamma^{-1}=\beta_j$, which is not possible since the $\beta_i$ and $\beta_j$ are non $\Gamma$-conjugated. Thus, $X(N)$ has at least $N^{1-\varepsilon'} \cdot \frac{|\Gamma/\Gamma(N)|}{2}$ elements. By Lemma \ref{prop:SN}(a) and Proposition~\ref{prop:translation length} any element in $X(N)$ induces a systole in $S_N$, and then
\[
\text{Kiss}(S_N) \geq N^{1-\varepsilon'} \cdot \frac{|\Gamma/\Gamma(N)|}{2}.
\]

Now, since
\(
\operatorname{area}(S_N) =|\Gamma/\Gamma(N)|\cdot \operatorname{area}(\Gamma \backslash \mathbb{H}^2) < v_0 N^3,\) with \(v_0 = \operatorname{area}(\Gamma \backslash \mathbb{H}^2)\) we get that
\[
\K(S_N) \geq \left( \frac{\operatorname{area}(S_N)}{v_0} \right)^{\frac{1-\epsilon'}{3}} \frac{\operatorname{area}(S_N)}{2v_0} \geq \frac{\operatorname{area}(S_N)^{\frac{4-\epsilon'}{3}}}{2v_0^2}=\frac{\operatorname{area}(S_N)^{\frac{4}{3}-\epsilon}}{2v_0^2}.
\]

On the other hand, for any \( \gamma \in \Gamma(N) \) hyperbolic, combining Proposition~\ref{prop:translation length} and Lemma~\ref{prop:SN} we obtain that
\begin{align*}
\ell_\gamma &\geq 2\log(|\tr(\gamma)| - 1) \geq 2\log(N^2 - 3) \geq 2\log\left(\frac{N^2}{2}\right)\\
&= 4\log(N) - 2\log(2) = \frac{4}{3}\log(N^3) - 2\log(2)\\
&>\frac{4}{3}\log\left( \frac{\operatorname{area}(S_N)}{\operatorname{area}(S)} \right) - 2\log(2)\\
&=\frac{4}{3}\log(\operatorname{area}(S_N)) - \left(2\log(2) + \frac{4}{3}\log(\operatorname{area}(S))\right).
\end{align*}

Thus,
\[
\operatorname{sys}(S_N) \geq \frac{4}{3}\log(\operatorname{area}(S_N)) - \left(2\log(2) + \frac{4}{3}\log(\operatorname{area}(S))\right). \qedhere
\]
\end{proof}

\subsection{Compact Case.}

Let $p\in\Z$ be a prime number such that $p\equiv 3\mod4$, and consider the quaternion algebra over $\Q$ given by
\[D_{p,-1} =\left(\frac{p,-1}{\Q}\right)=\left\{
x_0 + x_1 i + x_2 j + x_3 k \;\middle|\; x_i \in \mathbb{Q},\ i^2 = p,\ j^2 = -1,\ ij = -ji = k
\right\}
\]

By the discussion in Section~\ref{sec:Fuchsian groups} the group

\[
\Gamma = A_{a,b}^1(\mathbb{Z}) = \left\{
x \in A_{a,b} \;\middle|\;
x_i \in \mathbb{Z},\ x_0^2 - px_1^2 +x_2^2 -px_3^2 = 1
\right\}
\]
embeds as a cocompact Fuchsian group in $\SL_2(\R)$. For any $N\in\Z$ the \textit{principal congruence subgroup at level $N$} of $\Gamma$ is defined as
\[
\Gamma(N) = \left\{
x \in A_{p,-1}^1(\mathbb{Z}) \;\middle|\;
x_0 \equiv 1 \mod N,\ x_1, x_2, x_3 \equiv 0 \mod N
\right\}.
\]

In this case, we do not have a version of Theorem~\ref{th: Siegel}. Instead, we can use the Prime Geodesic Theorem for primitive closed geodesics (see \cite[Sec. 9.6]{Bus10}), which implies the existence of a sequence $\{t_i\} \subset \mathbb{N}$, $t_i \to \infty$ with
    \begin{equation}\label{eq:Prime Geodesic Thm}
    \#\{[x] \in \Gamma \mid \text{tr } x = t_i, \, x \text{ primitive} \} \sim \frac{t_i}{\log t_i},
    \end{equation}
where $[x]$ denotes the $\Gamma-$conjugacy class of $x$. The rest of the proof follows the same lines of argument as in the non-compact case.

\section{Systole, kissing number and volume in higher dimensions.}

\subsection{Arithmetic orbifolds and congruence coverings.}

We can resume the strategy in Section~\ref{sec:hyp surf larg sys and kiss} as follows.
Let $S=\Gamma\bs\H^2$ be an arithmetic hyperbolic surface defined over $\Q$, compact or non-compact. The length of closed geodesics in $S$ is related to the trace of hyperbolic elements in $\Gamma$ (Proposition~\ref{prop:translation length}), and the number of closed geodesics in $S$ with the same length relates to the number of conjugacy classes of hyperbolic elements in $\Gamma$ with the same trace. For any $N\in\Z_{>2}$ consider the (principal) congruence covering $S_N=\Gamma(N)\bs\H^2$ of $S$. The modular conditions defining $\Gamma(N)$ force $S_N$ to have a large systole; moreover, many non-homotopic curves realize this systole. That is, $\Gamma$ has a large number of non-conjugate hyperbolic elements with trace $N$ (Theorem~\ref{th: Siegel} and equation~\ref{eq:Prime Geodesic Thm}), all of them producing a systole on $S_N$ with large length Theorem~\ref{th:large kiss and systole nc}.

In order to keep the exposition not complicated, we do not plan to give the general definition of arithmetic orbifold and congruence covering. The technical details can be found in the articles that will be referenced below. Instead, we will give the main ideas, and hope that they will help the reader to follow the original arguments cited.

The fundamental group $\Gamma$ of an arithmetic orbifold $M$ has a representation as a matrix group $\Gamma=G(\mathcal{O}_k)$, with entries in the ring of integers $\mathcal{O}_k$ of a number field $k$, where $G$ is an algebraic $k$-subgroup of some $\GL_n(\C)$. The group $G$ is crucial, and its construction will depend on each type of orbifold considered. In our previous example, $k=\Q, \mathcal{O}_k=\Z$, and $ G=A^1_{a,b}$ where $A=\left(\frac{a,b}{\Q}\right)$ with $a,b\in\Q^\times, a>0$. For any ideal $I\subset\mathcal{O}_k$, the kernel of the reduction modulo $I$ map defines a finite index subgroup $\Gamma(I)<\Gamma$, the  \textit{principal congruence subgroup of $\Gamma$ at level $I$}, which induces a finite sheeted cover $M_I$ of $M$, called the \textit{principal congruence covering} of $M$, in the same way that $S_N$ was defined in Section~\ref{sec:hyp surf larg sys and kiss}.
A \textit{congruence subgroup} of $\Gamma$ is a subgroup $\Lambda <\Gamma$ containing some $\Gamma(I)$.

\subsection{Systole and volume in higher dimensions.}

Congruence coverings are natural covering spaces in any arithmetic locally symmetric space. We will mention the geometric results answering Question~\ref{q:q1} and Question~\ref{q:q2} in higher dimensions that have been obtained in the last few years, including other locally symmetric spaces, such as Hilbert modular varieties, arithmetic hyperbolic manifolds of the second type, complex and quaternionic hyperbolic manifolds, and special linear manifolds.

It is natural to expect that congruence coverings of arithmetic hyperbolic $n$-manifolds attain the logarithmic bound for their systole. In fact, comparing the geometry of the fundamental group of a compact hyperbolic manifold with the geometry of $\mathbb{H}^{n}$ it is possible to show that there are positive constants $C_{3}, C_{4}$ depending only on $n$ such that $$\sys(M_{I})\geq C_{3}\log(\vol(M_{I})) -C_{4}.$$

The proof can be found in \cite[Prop. 10]{GL14} (cf. \cite[3.C.6]{Gro96}). Although it shows that the systole of congruence coverings is bounded by a logarithmic function of the volume, the proof uses rough comparisons between the geometry of $\Gamma(I)$ and $\mathbb{H}^{n}$, and it does not give a precise value for the constants $C_{3}$ and $C_{4}.$

In dimension 3, M. Katz, M. Schaps, and U. Vishne found a precise lower bound for the systole  along congruence coverings

\begin{theorem}[see {\cite[Thm. 1.8]{KSV07}}]
    There exists a sequence $M_i$ of hyperbolic $3$-manifolds with $\vol(M_i)\to\infty$, such that
     $$\sys(M_{i})\geq\frac{2}{3}\log(\vol(M_{i}))-c,$$

 where $c$ is a constant that does not depend on $M_{i}.$
\end{theorem}

Extending the ideas in \cite{KSV07} to congruence coverings of arithmetic hyperbolic manifolds in higher dimensions, the following was proven by the author in \cite{Mur19}.

\begin{theorem}
    For any $n\geq 2$, there exists a sequence $M_i$ of compact hyperbolic $n$-manifolds with $\vol(M_i)\to\infty$, such that $$\sys(M_{i})\geq\frac{8}{n(n+1)}\log(\vol(M_{i}))-c,$$

    where $c$ is a constant that does not depend on $M_{i}.$
\end{theorem}

Together with C. Dória, we also proved that the constant $\frac{8}{n(n+1)}$ is optimal along principal congruence coverings (see the appendix of \cite{Mur19}). This implies that $C_1=\frac{8}{n(n+1)}$ holds as an answer to Question~\ref{q:q1} for compact hyperbolic $n$-manifolds.

Question~\ref{q:q1} has also been studied for other types of arithmetic locally symmetric spaces. For an arithmetic locally symmetric space of dimension $n$, there are congruence coverings $M_I$ of $M$ satisfying the inequality
$$\sys(M_{I})\geq C\log(\vol(M_{I}))-d,$$
where $d$ is a constant that does not depend on $M_{I}$. The constant $C$ depends on the type of $M$. We will summarize in Table~\ref{tab:sys} what is known so far about systole growth on other locally symmetric spaces. Also, we will indicate in which cases the provided $C$ is known to be optimal.

\begin{table}[!ht]
    \centering
    \resizebox{\textwidth}{!}{
\begin{tabular}{|c|c|c|c|}
    \hline
    $C$ & $M$ & Is $C$ optimal? & Reference \\
    \hline
    \(\frac{4}{3\sqrt{n}}\) & Hilbert modular variety & yes & Murillo \cite{M17} \\
    \hline
    \(\frac{4}{n(n+1)}\) & Real hyperbolic of type $II$ & not known & Kim \cite{K20} \\
    \hline
    \(\frac{2}{n(n+1)}\) & Complex hyperbolic & not known & Kim \cite{K20}\\
    \hline
    \(\frac{4}{3\sqrt{n}}\) & Compact quotients of $(\H^2)^n$ & not known &Cosac-Dória \cite{CD22}\\
    \hline
    \(\frac{4}{(n+1)(2n+3)}\) & Quaternionic hyperbolic manifold & yes & Emery- Kim-Murillo \cite{EKM22}\\
    \hline
    \(\frac{2\sqrt{2}}{n(n^2-2)}\) & Special linear manifold & not known & Lapan--Meyer \cite{LLM24}\\
    \hline
\end{tabular}
}
    \caption{Systole growth in different types of arithmetic orbifolds}
    \label{tab:sys}
\end{table}

\subsection{Kissing number in higher dimensions.}\label{sec:kiss high dim}
As it was mentioned in Section~\ref{sec:two guiding questions}, an answer to Question~\ref{q:q1} can be obtained from an answer to Question~\ref{q:q2}, but the latter is more difficult. The first construction of hyperbolic $3$-manifolds with large kissing number was obtained by this author in collaboration with C. Dória. More precisely, the following result was obtained in \cite{DM21}.

\begin{theorem}\label{th:kiss dim 3 nc}
There exists a sequence $M_i$ of non-compact finite volume hyperbolic $3$-manifolds with $\vol(M_i)\to\infty$, such that \[\K(M_i)\geq c\cdot\frac{\vol(M_i)^{\frac{31}{24}}}{\log(\vol(M_i))}\]
 for some constant $c>0$ independent of $M_i$.
\end{theorem}

A Bianchi orbifold is the quotient of $\H^3$ by $\PSL(\mathcal{O}_D)$ for some $D>0$, where $\mathcal{O}_D$ denotes the ring of integers of the quadratic imaginary number field $\Q(\sqrt{-D})$. These orbifolds can be seen as models for non-compact finite volume arithmetic hyperbolic 3-manifolds, and generalize the modular surface discussed in Section~\ref{sec: Modular group}. In fact, for any non-compact finite volume arithmetic hyperbolic $3$-manifold, there is a Bianchi orbifold with a finite-sheeted hyperbolic cover in common (see \cite[Thm 8.2.3]{MR02}). The manifolds $M_i$ in Theorem \ref{th:kiss dim 3 nc} are congruence coverings (not necessarily principal) of Bianchi orbifolds. Later on, together with C. Dória and E. Freire, we extended Theorem \ref{th:kiss dim 3 nc} to compact hyperbolic $3$-manifolds, and also improved the constant $\frac{31}{24}$, see \cite{DFM24}. In the same work, we prove the following using congruence coverings of certain compact arithmetic hyperbolic orbifolds of higher dimension.

\begin{theorem}
 For any $n\geq 2$, there exists a sequence of compact hyperbolic $n$-manifolds $M_i$ with $\vol(M_i)\to\infty$ such that \[\K(M_j)\geq c\frac{\vol(M_j)^{1+\frac{1}{3n(n+1)}}}{\log(\vol(M_j))}\]
 for some constant $c> 0$ independent of $M_j$.
\end{theorem}

\subsection{Open questions.}

We would like to finish the exposition with some open questions that complement Questions~\ref{q:q1} and~\ref{q:q2}.

\begin{question}
In most of the cases in Table~\ref{tab:sys}, the constant $C$ is expected to be not optimal, except possibly for compact quotients of $\left(\H^2\right)^n$. What is the best possible constant $C$ for the locally symmetric spaces in those constructions?
\end{question}

\begin{question}
Is there a non-compact version of Theorem~\ref{th:BP}? We are not aware of this result even for $n=3$. It is expected that a careful analysis of the techniques in Theorem~\ref{th:BP} holds.
\end{question}

\begin{question}
    Is it possible to get a version of Theorem~\ref{th:BP} for locally symmetric spaces, compact or not, which are not hyperbolic?
\end{question}

\begin{question}
    To the best of our knowledge, the results in Section~\ref{sec:kiss high dim} are the only known results that give an answer to Question~\ref{q:q2} in any dimension. Is it possible to obtain similar results for other types of locally symmetric spaces, which are not hyperbolic?
\end{question}

\begin{question}
    Suppose that $M$ is a hyperbolic manifold satisfying inequality \eqref{eq: M with large kiss} for some $\epsilon>0$. Is $M$ necessarily arithmetic?
\end{question}

\end{document}